\documentclass[a4paper,10pt]{article}

\usepackage{amsmath}
\usepackage{amssymb}
\usepackage{amsthm}
\usepackage{mathrsfs}
\usepackage{color}
\usepackage[dvips]{graphicx}
\usepackage[all]{xy}

\theoremstyle{plain}
\newtheorem{them}{Theorem}[section]
\newtheorem{lemma}[them]{Lemma}
\newtheorem{prop}[them]{Proposition}
\newtheorem{coro}[them]{Corollary}
\theoremstyle{definition}
\newtheorem{defi}[them]{Definition}
\newtheorem{exam}[them]{Example}

\newtheorem{conj}[them]{Conjecture}

\newtheorem*{mthem}{Main Theorem}

\newcommand{\map}[3]{#1:#2 \rightarrow #3}

\begin{document}

\title{The zeta function of a finite category}
\author{Kazunori Noguchi \thanks{noguchi@math.shinshu-u.ac.jp}}
\date{}
\maketitle
\begin{abstract}
 We define the zeta function of a finite category. And we propose a conjecture which states the relationship between the Euler characteristic of finite categories and the zeta function of finite categories. This conjecture is verified when categories are finite groupoids, finite acyclic categories, categories with 2-objects and finite categories satisfying certain condition.
\end{abstract}

\tableofcontents


\thispagestyle{empty}

\section{Introduction}

The Euler characteristic and the zeta function are defined for various mathematical objects, for example simplicial complexes, algebraic varieties, graphs, and so on. In many cases, we can see the zeta function knows the Euler characteristic, as the following three examples suggest.
\begin{enumerate}
\item Let $G$ be a finite graph. Then, Ihara zeta function of $G$ is defined by $$Z_G(u)=\prod_{[C]}\frac{1}{1-u^{l[C]}}$$
where $[C]$ is an equivalence class of certain paths in $G$ and $l$ is the length function. 
The zeta function $Z_G$ has the determinant expression $$Z_G(u)=\frac{(1-u^2)^{\chi(G)}}{\det(E-A_Gu+Q_Gu^2)}$$
for some matrices $A_G,Q_G$ where $\chi(G)$ is the Euler characteristic of $G$ \cite{ST96}.

\item Let $\Delta$ be a simplicial complex on vertex set $$\{1,2,\dots, n\}$$ and let $\mathbb{F}_q$ be a finite field. Bj\"orner and Sarkaria defined the zeta function of $\Delta$ over $\mathbb{F}_{q}$ by $$Z_{\Delta}(q,t)=\exp \left( \sum_{k=1}^{\infty} \#V(\Delta, \mathbb{F}_{q^k})\frac{t^k}{k}\right)$$
where $V(\Delta, \mathbb{F}_{q^k})$ is the set of points in the projective space $\mathbb{F}_{q^k}P^{n-1}$ whose support belongs to $\Delta$.
The zeta function has a rational expression, that is,
$$Z_{\Delta}(q,t)=\prod_{j=0}^n \frac{1}{(1-q^jt)^{f^*_j}}.$$
Here, we obtain $\sum^n_{j=0} f^*_j=\chi(V(\Delta, \mathbb{C}))$ \cite{BS98}.

\item Let $X$ be an $n$-dimensional smooth projective variety over a finite field $\mathbb{F}_q$. Then, the zeta function of $X$ is defined by $$Z_X(T)=\exp\left( \sum_{m=1}\frac{N_m(X)}{m}T^m \right)$$
where $N_m(X)$ is the number of points in $X$ over $\mathbb{F}_{q^m}$. The Weil conjecture (Deligne's theorem) states that $Z_X$ has the rational expression of the following form
$$Z_X(T)=\frac{P_1(T)\dots P_{2n-1}(T)}{P_0(T)\dots P_{2n}(T)}$$
where each $P_i(T)$ is a polynomial with the coefficient in $\mathbb{Z}$ and we obtain $\chi(X)=\sum_{i=0}^{2n}(-1)^i\deg P_i(T)$ \cite{Har77}.
\end{enumerate}
These examples tell us that the zeta function knows the Euler characteristic. 

The zeta functions are defined for finite graphs and finite directed graphs. And finite categories are similar to finite graphs and finite directed graphs since the three notions are consisted by vertices (objects) and edges (morphisms). Therefore, we can expect that the zeta function is defined for finite categories. And three examples above tells us that the zeta function of finite categories may know the Euler characteristic of finite categories.

In this paper, we investigate the Euler characteristic of finite categories from the view point of the zeta function of finite categories.

First let us recall the Euler characteristic of categories. In \cite{Leia}, the Euler characteristic of a finite category was defined. This is the first Euler characteristic for categories and a few works come after it, the series Euler characteristic \cite{Leib}, the $L^2$-Euler characteristic \cite{FLS}, the extended $L^2$-Euler characteristic \cite{Nog} and the Euler characteristic for $\mathbb{N}$-filtered acyclic categories \cite{Nog11}. In this paper, we often use the series Euler characteristic, so we give more detail explanation for the series Euler characteristic.

For a finite category $I$ whose set of objects is $\{x_1,\dots, x_n\}$, its series Euler characteristic $\chi_{\Sigma}(I)$ is defined by substituting $-1$ to $t$ of 
$$\frac{\mathrm{sum}(\mathrm{adj}(E-(A_I-E)t))}{\det(E-(A_I-E)t)}$$
if it exists where $A_I=(\#\mathrm{Hom}(x_i,x_j))_{i,j}$ is the adjacency matrix of $I$ and sum means to take the sum of all of the entries of a matrix. This rational function is the rational expression of the power series $\sum_{n=0}^{\infty} \#\overline{N_n}(I)t^n$ where $\overline{N_n}(I)$ is the set of non-degenerate chains of morphisms whose length is $n$ in $I$
$$\overline{N_n}(I)=\{\xymatrix{(x_0\ar[r]^{f_1}&x_1\ar[r]^{f_2}&\dots\ar[r]^{f_n}&x_n)} \text{ in } I \mid  f_i \not = 1\}.$$
 This Euler characteristic is defined by the view point of the classifying spaces. For a small category $J$, we can construct a topological space (indeed, a CW-complex) $BJ$, called the classifying space of $J$. There is a one-to-one correspondence between the set of $n$-dimensional parts ($n$-cells) of $BJ$ and $\overline{N_n}(I)$. The Euler characteristic of a cell-complex is defined by the alternating sum of the number of $n$-cells. So the Euler characteristic of $I$ should be defined by $\sum_{n=0}^{\infty}(-1)^n\# \overline{N_n}(I)$. But this series often fails to converge, so we substitute $-1$ to $t$ of the rational expression instead of the power series $\sum_{n=0}^{\infty} \#\overline{N_n}(I) t^n$. When there exists the inverse matrix of $A_I$, Leinster's Euler characteristic and the series Euler characteristic coincide and they are equal to sum of all the entries of $A_I^{-1}$ \cite{Leib}.

The zeta function of finite categories is introduced in this paper. This is different from the one introduced by Kurokawa \cite{Kur96}. His zeta function is for a "large" category, for example the category of Abelian groups. We do not use the definition for finite categories. 

Let $I$ be a finite category. Then, the zeta function of $I$ is defined by $$\zeta_I(z)=\exp\left( \sum_{m=1}^{\infty} \frac{\# N_m(I)}{m} z^m\right)$$ where $$N_m(I)=\{\xymatrix{(x_0\ar[r]^{f_1}&x_1\ar[r]^{f_2}&\dots\ar[r]^{f_n}&x_n)} \text{ in } I \}.$$
The relationship between the zeta function of finite categories and the Euler characteristic of finite categories is summarized in the following conjecture.

\begin{conj}\label{noguchi}
Suppose $I$ is a finite category and its series Euler characteristic exists. Then, we have

\renewcommand{\theenumi}{C\arabic{enumi}}
\renewcommand{\labelenumi}{(\theenumi)}
\begin{enumerate}
\item\label{c1} the zeta function of $I$ is a finite product of the following form $$\zeta_I(z)=\prod \frac{1}{(1-\alpha_iz)^{\beta_i}}\exp\left( \sum \frac{\gamma_j z^j}{j(1-\delta_j z)^j}\right)$$
for some complex numbers $\alpha_i, \beta_i, \gamma_j,\delta_j $.
\item\label{c2}  $\displaystyle \sum \beta_i$ is the number of objects of $I$ 
\item\label{c3} each $\alpha_i$ is an eigen value of $A_I$. Hence, $\alpha_i$ is an algebraic integer.
\item\label{c4} $\displaystyle \sum \frac{\beta_i}{\alpha_i} +\sum (-1)^j\frac{\gamma_j}{\delta_j^{j+1}}=\chi_{\Sigma}(I).$
\end{enumerate}
\end{conj}
For certain class of finite categories, this conjecture is verified, more precisely
\begin{mthem}
Let $I$ be a finite category. Then, Conjecture \ref{noguchi} holds true if $I$ satisfies one of the following conditions
\begin{enumerate}
\item $I$ is a groupoids
\item $I$ is an acyclic category
\item $I$ has two objects and $\chi_{\Sigma}(I)$ exists
\item the adjacency matrix $A_I=(a_{ij})$ satisfies $\sum_j a_{ij}=\sum_j a_{i'j}$ for any $i,i'$
\item the adjacency matrix $A_I=(a_{ij})$ satisfies $\sum_i a_{ij}=\sum_i a_{ij'}$ for any $j,j'$
\item the adjacency matrix is $\begin{pmatrix}2&2&2\\2&2&2\\2&8&5\end{pmatrix}$ or $\begin{pmatrix}2&3&2 \\1&2&6 \\1&1&2\end{pmatrix}$ or $\begin{pmatrix}4&7&8 \\1&4&5 \\1&1&3\end{pmatrix}$.
\end{enumerate}
\end{mthem}

We give some explanations about this conjecture.

For (\ref{c3}), the part that $\alpha_i$ is an algebraic integer is an analogue of the Weil conjecture. If $\alpha_i$ is an eigen value of $A_I$, then $\alpha_i$ is an algebraic integer. Indeed, all the entries of $A_I$ are non-negative integers, so $\det(E\lambda-A_I)$ is a monic polynomial with coefficients in $\mathbb{Z}$.

For (\ref{c4}), the reason the series Euler characteristic appears in this conjecture will be explained in \S \ref{many-exam}. There exists a category whose series Euler characteristic and Leinster's Euler characteristic are defined and they are not equal. For this category, the left hand side of $(\ref{c4})$ is  the series Euler characteristic. We have other example whose Leinster's Euler characteristic is defined, but the series Euler characteristic does not. For this category, the left hand side of (\ref{c4}) does not coincide Leinster's Euler characteristic. The Euler characteristic of $\mathbb{N}$-filtered acyclic categories can not be used for general finite category. And the $L^2$-Euler characteristic is not determine by only the matrix $A_I$, but the zeta function is determined since $\#N_m(I)=\mathrm{sum}(A_I^m)$. So we use the series Euler characteristic in this conjecture. 

Moreover, we note the left hand side of (\ref{c4}) uses the complex numbers. Almost Euler characteristic are defined by the integers, but the Euler characteristics of finite categories are defined by rational numbers in various ways. For infinite categories, the $L^2$-Euler characteristic is defined by the real numbers. (The complex numbers are used as a field in its definition, but they are not used as numbers.) It is remarkable that the Euler characteristic is computed via complex numbers (see Example \ref{complex1},\ref{complex2}).

This paper is organized as follows.
In section \ref{zeta}, the zeta function of a finite category is defined. And we prove Conjecture \ref{noguchi} holds true for finite groupoids, finite acyclic categories, finite categories with 2-objects and finite categories satisfying certain condition. Furthermore, we prove for a covering of groupoids $\map{P}{E}{B}$, $\zeta^{-1}_B(z)$ divides $\zeta^{-1}_E(z)$. In the last of this section, we introduce four examples of our zeta function. In section \ref{graph}, we discover the relation between the zeta function of directed graph and the zeta function of finite categories.

\section{The zeta function of a finite category}\label{zeta}
\subsection{Definition}
Before giving the definition of the zeta function of a finite category, we review the symbols which are often used in this paper.

Let $I$ be a finite category. Then, let
$$N_n(I)=\{\xymatrix{(x_0\ar[r]^{f_1}&x_1\ar[r]^{f_2}&\dots\ar[r]^{f_n}&x_n)} \text{ in } I \}$$
and
$$\overline{N_n}(I)=\{\xymatrix{(x_0\ar[r]^{f_1}&x_1\ar[r]^{f_2}&\dots\ar[r]^{f_n}&x_n)} \text{ in } I \mid  f_i \not = 1\}.$$
The difference between them is just one thing that the identity morphisms are used or not. For $n=0$, we set $N_0(I)=\overline{N_0}(I)=\mathrm{Ob}(I)$.
\begin{defi}
Let $I$ be a finite category. Then, define \textit{the zeta function} $\zeta_I(z)$ of $I$ by
$$\zeta_I(z)=\exp\left( \sum_{m=1}^{\infty} \frac{\# N_m(I)}{m} z^m\right).$$ 
The symbol $z$ is a formal variable. If one prefers, the zeta function can be considered as a function of a complex variable by choosing $z$ to be a sufficiently small complex number.
\end{defi}

\begin{exam}
This is the simplest example. Let $*$ be the one-point category. Then, its zeta function is
\begin{eqnarray*}
\zeta_*(z)&=&\exp\left( \sum_{m=1}^{\infty} \frac{\# N_m(*)}{m} z^m\right) \\
&=&\exp\left( \sum_{m=1}^{\infty} \frac{1}{m} z^m\right) \\
&=&\exp\left( -\log (1-z) \right) \\
&=& \frac{1}{1-z}.
\end{eqnarray*}
\end{exam}

\subsection{Groupoids}

In this subsection, we prove the conjecture \ref{noguchi} holds true for finite groupoids and $\zeta^{-1}_B(z)$ divides $\zeta^{-1}_E(z)$ for a covering of groupoids $\map{P}{E}{B}$.

\begin{prop}\label{connected}
Let $\Gamma$ be a connected groupoid. Then, its zeta function is
$$\zeta_{\Gamma}(z)=\frac{1}{(1-\# N_0(\Gamma) o(\Gamma)z)^{\# N_0(\Gamma)}}$$
where $o(\Gamma)$ is the order of the automorphism group $\mathrm{Aut}(x)$ for some object $x$ of $\Gamma$.
\begin{proof}
Let $$\mathrm{Ob}(\Gamma)=\{x_1,x_2,\dots, x_n\}.$$ We count how many chains of morphisms whose length is $m$ there are in $\Gamma$. To determine $$\mathbf{g}=\xymatrix{(y_0\ar[r]^{f_1}&y_1\ar[r]^{f_2}&\dots\ar[r]^{f_m}&y_m)}$$
we first determine objects $y_0, y_1,\dots, y_m$. There are $n^{m+1}$ ways to choice the objects. And there are $o(\Gamma)^m$-ways to choice morphisms $f_1,f_2,\dots, f_m$ since we have $$\# \mathrm{Hom}(x,y)=\# \mathrm{Hom}(x',y')=o(\Gamma)$$ for any objects $x,x',y,y'$ of $\Gamma$. Hence we obtain $\# N_m(\Gamma)=n^{m+1} o(\Gamma)^m$. Thus, we have
\begin{eqnarray*}
\zeta_{\Gamma}(z)&=&\exp\left( \sum_{m=1}^{\infty} \frac{ n^{m+1} o(\Gamma)^m  }{m} z^m\right) \\
&=&\exp\left( n \sum_{m=1}^{\infty} \frac{ 1 }{m} (n o(\Gamma) z)^m\right) \\
&=& \exp\left( -n \log (1-n o(\Gamma) z)\right) \\
&=& \frac{1}{(1-n o(\Gamma)z)^{n}}.
\end{eqnarray*}
\end{proof}
\end{prop}

\begin{lemma}\label{coprod}
Let $I_1,I_2,\dots, I_n$ be finite categories. Then, the zeta function of $I=\coprod_{i=1}^n I_i$ is
$$\zeta_I(z)=\prod_{i=1}^n \zeta_{I_i}(z).$$
\begin{proof}
Since $N_m(I)=\coprod_{i=1}^n N_m(I_i)$, we obtain
\begin{eqnarray*}
\zeta_I(z)&=&\exp\left( \sum_{m=1}^{\infty} \frac{\# N_m(I)}{m} z^m\right) \\
&=&\exp\left( \sum_{m=1}^{\infty} \frac{\# N_m(I_1)+\# N_m(I_2)+\dots +\# N_m(I_n)}{m} z^m\right) \\
&=&\prod_{i=1}^n\exp\left( \sum_{m=1}^{\infty} \frac{\# N_m(I_i)}{m} z^m\right) \\
&=&\prod_{i=1}^n \zeta_{I_i}(z).
\end{eqnarray*}
\end{proof}
\end{lemma}

\begin{coro}\label{groupoid}
Suppose $\Gamma$ is a finite groupoid and $\Gamma_1, \Gamma_2,\dots,\Gamma_n$ are its connected components, that is, $\Gamma=\coprod_{i=1}^n \Gamma_i$ and each $\Gamma_i$ is connected. Then, the zeta function of $\Gamma$ is
$$\zeta_{\Gamma}(z)=\prod_{i=1}^n \frac{1}{(1-\# N_0(\Gamma_i) o(\Gamma_i)z)^{\# N_0(\Gamma_i)}}.$$
\begin{proof}
Lemma \ref{coprod} and Proposition \ref{connected} directly imply this.
\end{proof}
\end{coro}

\begin{them}
Under the same assumption of \ref{groupoid}, the conjecture \ref{noguchi} holds true.
\begin{proof}
By corollary \ref{groupoid} (\ref{c1}) holds true. It is clear that (\ref{c2}), (\ref{c3}) are satisfied. And we have $$\sum^n_{i=1}\frac{\#N_0(\Gamma_i)}{\#N_0(\Gamma_i) o(\Gamma_i)}=\sum_{i=1}^n\frac{1}{o(\Gamma_i)}=\chi_{\Sigma}(\Gamma).$$ 
The last equality is implied by Theorem 3.2 of \cite{Leib} and Example 2.7 of \cite{Leia}.
\end{proof}
\end{them}

The rest of this subsection is devoted to the proof that $\zeta^{-1}_B(z)$ divides $\zeta^{-1}_E(z)$ for a covering of groupoids $\map{P}{E}{B}$. This is an analogue of relation between the zeta function of finite graphs and coverings of graphs. Let $X,Y$ be connected $(q+1)$-regular graphs and let $\map{p}{Y}{X}$ be a covering of graphs. Then, $Z_X(u)^{-1}$ divides $Z^{-1}_Y(u)$ \cite{ST96}. The similar fact holds for a covering of groupoids and the zeta function of groupoids. 

Let $E,B$ be a connected groupoids. A \textit{covering of groupoids} is a functor $\map{P}{E}{B}$ such that $P$ is a surjection on the set of objects and a restriction to $\mathrm{St}(e)$ and $\mathrm{St}(P(e))$ is a bijection for any object $e$ of $E$ where $\mathrm{St}(e)$ is the set of morphisms in $E$ from $e$ \cite{May99}.

\begin{prop}
Suppose $E,B$ are connected finite groupoids and $\map{P}{E}{B}$ be a covering of groupoids. Then, $\zeta_{B}^{-1}$ divides $\zeta_E^{-1}$.
\begin{proof}
Proposition \ref{connected} implies $$\zeta_{B}(z)=\frac{1}{(1-\# N_0(B) o(B)z)^{\# N_0(B)}}, \zeta_{E}(z)=\frac{1}{(1-\# N_0(E) o(E)z)^{\# N_0(E)}}.$$ For an object $e$ of $E$, we have
\begin{eqnarray*}
\# \mathrm{St}(e)&=&\sum_{x\in \mathrm{Ob}(E)}\# \mathrm{Hom}_E(e,x) \\
&=&\sum_{x\in \mathrm{Ob}(E)} o(E) \\
&=&\#N_0(E) o(E).
\end{eqnarray*}
Since $P$ is a covering of groupoids, we have $\mathrm{St}(e)=\mathrm{St}(P(e))=\#N_0(B)o(B)$. And $P$ is a surjection on the set of objects, so $\#N_0(B) \le \# N_0(E)$. Hence, $\zeta_{B}^{-1}$ divides $\zeta_E^{-1}$.

\end{proof}
\end{prop}

\subsection{Acyclic categories}

In this subsection, we prove Conjecture \ref{noguchi} holds true for finite acyclic categories by using another expression of our zeta function.

\begin{defi}
Define a small category $A$ to be an \textit{acyclic category} if all the endomorphisms are only identity morphisms and if there exists an arrow $\map{f}{X}{Y}$ such that $X\not = Y$, then there does not exist an arrow $\map{g}{Y}{X}$. 
\end{defi}

We have another expression of the zeta function by non-degenerate nerves.

\begin{prop}\label{non-degenerate}Let $I$ be a finite category.
Then, we have
$$\zeta_I(z)=\frac{1}{(1-z)^{\# N_0(I)}} \exp\left( \sum_{k=1}^{\infty} \frac{\# \overline{N_k}(I) }{k}(z^{-1}-1)^{-k}  \right)$$
\end{prop}

\begin{lemma}\label{exchange}
Let $I$ be a finite category. Then, we have $$\# N_m(I)=\sum_{i=0}^m \binom{m}{i}\# \overline{N_i}(I)$$ for any $0 \le m$.
\begin{proof}
Suppose $0\le i\le m$ and take any $\mathbf{f}=(f_1,f_2,\dots,f_i)$ of $\overline{N_i}(I)$. Then, we can make $\binom{m}{i}$-elements of $N_m(I)$ by inserting the identity morphisms. Hence, we obtain this result.
\end{proof}
\end{lemma}

\begin{lemma}\label{P-Lemma}
Fix a natural number $k$. For any $0\le n$ we have
$$k\left(\sum^{k-1+n}_{j=k-1} {}_{j}P_{k-1}\right)={}_{k+n}P_k$$
where ${}_lP_m=l(l-1)\dots (l-(m-1))$.
\begin{proof}
We prove this by induction on $n$. 

When $n=0$, we have
\begin{eqnarray*}
k{}_{k-1} P_{k-1}=k(k-1)!=k!={}_kP_k.
\end{eqnarray*}
We suppose the equality holds at $n$. For $n+1$, we have
\begin{eqnarray*}
k\left(\sum^{k-1+(n+1)}_{j=k-1} {}_{j}P_{k-1}\right)&=&k\left(\sum^{k-1+n}_{j=k-1} {}_{j}P_{k-1}\right)+k{}_{k+n}P_{k-1} \\
&=&{}_{k+n} P_k +k{}_{k+n}P_{k-1}\\
&=&(k+n)(k+n-1)\dots (k+n-(k-1))\\
&&+k(k+n)(k+n-1)\dots (k+n-(k-2)) \\
&=&(k+n)(k+n-1)\dots (n+1) \\
&&+k(k+n)(k+n-1)\dots (n+2) \\
&=&(k+n+1)\{(k+n)\dots (n+2)\} \\
&=&{}_{k+n+1}P_{k}.
\end{eqnarray*}
\end{proof}
\end{lemma}

\begin{lemma}\label{differential}
Let $k$ be a natural number. Then, we have $$ \sum_{m=1}^{\infty} {}_{m-1}P_k z^m = \frac{k!z^{k+1}}{(1-z)^{k+1}}$$
\begin{proof}
We prove this by induction on $k$. 

When $k=1$, we have
\begin{eqnarray*}
\sum_{m=1}^{\infty}(m-1)z^m &=&\sum_{m=1}^{\infty} mz^m - \sum_{m=1}^{\infty}z^m \\
&=&(z\frac{d}{dz})\sum_{m=1}^{\infty}z^m -\sum_{m=1}^{\infty}z^m \\
&=&(z\frac{d}{dz})\frac{z}{1-z} -\frac{z}{1-z} \\
&=&\frac{z^2}{(1-z)^2}
\end{eqnarray*}
Suppose the equality holds for $k-1$. Then for $k$ we have
\begin{eqnarray*}
\frac{k!z^{k+1}}{(1-z)^{k+1}}&=&\frac{(k-1)!z^k}{(1-z)^{k}}\frac{k z}{(1-z)} \\
&=&\left(\sum^{\infty}_{m=1} {}_{m-1}P_{k-1} z^m\right)k\sum^{\infty}_{m=1}z^m \\
&=&\sum^{\infty}_{m=2}\left( k\sum^{m-1}_{j=k} {}_{j-1}P_{k-1}\right)z^m \\
&=&\sum^{\infty}_{m=2}\left( k\sum^{k-1+(m-k)}_{j=k} {}_{j-1}P_{k-1}\right)z^m \\
&=&\sum^{\infty}_{m=2}\left( k\sum^{k-1+(m-k-1)}_{j=k-1} {}_{j}P_{k-1}\right)z^m \\
&=&\sum^{\infty}_{m=2} {}_{k+(m-k-1)}P_k z^m \\
&=&\sum^{\infty}_{m=2} {}_{m-1}P_k z^m (\text{By Lemma \ref{P-Lemma}})\\
&=&\sum^{\infty}_{m=1} {}_{m-1}P_k z^m.
\end{eqnarray*} 
\end{proof}
\end{lemma}

\begin{proof}[Proof of Proposition \ref{non-degenerate}]



\begin{eqnarray*}
\zeta_I(z)&=&\exp\left( \sum_{m=1}^{\infty} \frac{1}{m} \#N_m(I) z^m \right) \\
&=&\exp\left( \sum_{m=1}^{\infty} \frac{1}{m} \sum_{k=0}^m  \binom{m}{k}\# \overline{N_k}(I) z^m \right) (\text{ By Lemma \ref{exchange}})\\
&=&\exp\left( \sum_{k=0}^{\infty} \# \overline{N_k}(I) \sum_{m=1}^{\infty} \frac{1}{m}\binom{m}{k} z^m \right) \\
&=&\exp\bigg( \# N_0(I)\sum^{\infty}_{m=1}\frac{1}{m}z^m+ \\
&&\sum_{k=1}^{\infty} \# \overline{N_k}(I) \sum_{m=1}^{\infty} \frac{1}{m}\binom{m}{k} z^m \bigg) (\text{By Lemma \ref{differential}}) \\
&=&\exp\left(-\# N_0(I)\log(1-z) + \sum_{k=1}^{\infty} \frac{\# \overline{N_k}(I)}{k(z^{-1}-1)^k}  \right) \\
&=&\frac{1}{(1-z)^{\#N_0(I)}}\exp\left(\sum_{k=1}^{\infty} \frac{\# \overline{N_k}(I)}{k(z^{-1}-1)^k}  \right) .
\end{eqnarray*}

\end{proof}

\begin{them}
Let $A$ be a finite acyclic category. Then Conjecture \ref{noguchi} holds true.
\begin{proof}

Proposition \ref{non-degenerate} implies (\ref{c1}) is satisfied, that is, we obtain 
$$\zeta_A(z)=\frac{1}{(1-z)^{\#N_0(A)}}\exp\left(\sum_{k=1}^{\dim(A)} \frac{\# \overline{N_k}(A)}{k(z^{-1}-1)^k}  \right)$$
where $\dim(A)=\max\{n \mid n\ge 0, \overline{N_n}(A)\not =\emptyset \}$.
It is clear (\ref{c2}) is satisfied and (\ref{c3}) is also since the adjacency matrix of $A$ is an upper triangular matrix whose diagonal entries are all 1. Furthermore, we obtain $$\frac{\# N_0(A)}{1}+\sum_{k=1}^{\dim(A)}(-1)^k\frac{\#\overline{N_k}(A)}{1^{k+1}}=\chi_{\Sigma}(I).$$
The last equality is implied by Theorem 3.2 of \cite{Leib} and Corollary 1.5 of \cite{Leia}.
\end{proof}
\end{them}


\subsection{Categories which have two objects}

In this subsection, we give a classification of the zeta function of finite categories when categories have exactly two objects. All the zeta function we have already seen are expressed by the rational numbers, but the real numbers appear in the classification. Here, we take one example with no proof. (This follows from Theorem \ref{two}.)

 Let $\mathbb{F}$ be the following category  
$$\xymatrix{A\ar@<1ex>[r]^i&X\ar@<1ex>[l]^r \ar@(ur,dr)[]|{i\circ r}}$$
where $r\circ i=1_A, r \circ i\not = 1_X$. Then, $A_{\mathbb{F}}=\begin{pmatrix}1&1\\ 1&2\end{pmatrix}$. The zeta function of $\mathbb{F}$ is
$$\zeta_{\mathbb{F}}(z)=\frac{1}{\left( 1-\left( \frac{3+\sqrt{5}}{2} \right) z\right)^{1+\frac{2}{\sqrt{5}}}} \frac{1}{\left( 1-\left( \frac{3-\sqrt{5}}{2} \right) z\right)^{1-\frac{2}{\sqrt{5}}}}.$$

The reason that $\sqrt{5}$ appears is the sequence $(\#N_m(I))_{m\ge 0}$ is a subsequence of the Fibonacci sequence $(F_m)_{m\ge 1}$, that is, we have $\#N_m(I)=F_{m+3}$ and each $$F_m=\frac{1}{\sqrt{5}}\left( \bigg( \frac{1+\sqrt{5}}{2}\bigg)^n -\bigg( \frac{1-\sqrt{5}}{2}\bigg)^n\right).$$ Here, we have some remarkable points with respect to Conjecture \ref{noguchi}.
\begin{enumerate}
\item The numbers $\frac{3\pm \sqrt{5}}{2} $ are the eigen value of $A_{\mathbb{F}}$ and the algebraic integers. More precisely, they are the integers in the real quadratic number field $\mathbb{Q}(\sqrt{5})$. The set of the integers in $\mathbb{Q}(\sqrt{5})$ is $$ \biggl\{ \frac{a+b\sqrt{5}}{2} \bigg| a, b\in \mathbb{Z}, a \equiv b \mod 2 \biggr\}.$$ As we will see, all the coefficients of $z$ belong to the real quadratic number fields when a category has two objects. If we drop the hypothesis $\# \mathrm{Ob}=2$, all such coefficients do not belong to the real quadratic number fields.
\item The sum of the indices is the number of objects in $\mathbb{F}$ $$\bigg( 1+\frac{2}{\sqrt{5}}\bigg)+\bigg(1-\frac{2}{\sqrt{5}}\bigg)=2.$$
\item We obtain $$\frac{1+\frac{2}{\sqrt{5}}}{\frac{3+\sqrt{5}}{2}}+\frac{1-\frac{2}{\sqrt{5}}}{\frac{3-\sqrt{5}}{2}}=1=\mathrm{sum}(A_{\mathbb{F}}^{-1})=\chi_{\Sigma}(\mathbb{F}).$$
\end{enumerate}

\begin{them}\label{two}
Let $I$ be a finite category and let $A_I=\begin{pmatrix} a&b \\ c&d\end{pmatrix}$.
\begin{enumerate}
\item If $b=c=0$, we have\label{dis}
$$\zeta_I(z)=\frac{1}{(1-az)}\frac{1}{(1-dz)}$$
\item when $c=0, b\not =0$, 
\begin{enumerate}
\item \label{2a}if $a\not =d$, we have
$$\zeta_I(z)=\frac{1}{(1-az)^{1-\frac{b}{d-a}}}\frac{1}{(1-dz)^{1+\frac{b}{d-a}}}$$
\item if $a=d$, we have
$$\zeta_I(z)=\frac{1}{(1-az)^2}\exp \left( \frac{b}{z^{-1}-a}\right)$$
\end{enumerate}
\item when $b,c\not =0$
\begin{enumerate}
\item \label{det}if $\det A_I\not=0$, we have
$$\zeta_I(z)=\frac{1}{(1-B^{+}z)^{1+\frac{b+c}{\sqrt{(d-a)^2+4bc}}}}\frac{1}{(1-B^- z)^{1-\frac{b+c}{\sqrt{(d-a)^2+4bc}}}}$$
where 
$$B^{\pm}=\frac{(d+a)\pm\sqrt{(d-a)^2+4bc}}{2}$$
\item if $\det A_I =0$, we have
$$\zeta_I(z)=\frac{1}{(1-(a+d)z)^{\frac{a+b+c+d}{a+d}}}.$$
\end{enumerate}
\end{enumerate}

\begin{proof}
\begin{enumerate}
\item In this case, the category $I$ consists of two categories. So we have $\#N_m(I)=a^m+b^m$. Hence, we obtain $$\zeta_I(z)=\frac{1}{(1-az)}\frac{1}{(1-dz)}.$$
\item \begin{enumerate}
\item The same proof of (\ref{det}) can be used for this case. The proof is given in the following. 
\item We have 
\begin{eqnarray*}
\#N_m(I)&=&a^m+a^{m-1}b+a^{m-2}bd+\dots +abd^{m-2}+bd^{m-1}+d^m \\
&=&2a^m+m ba^{m-1}.
\end{eqnarray*}
Hence, we obtain 
\begin{eqnarray*}
\zeta_I(z)&=&\exp\left(\sum^{\infty}_{m=1} \frac{1}{m}(2a^m+mba^{m-1})z^m\right) \\
&=&\frac{1}{(1-az)^2}\exp \left( \frac{b}{a} \frac{az}{1-az}\right) \\
&=&\frac{1}{(1-az)^2}\exp \left( \frac{b}{z^{-1}-a}\right).
\end{eqnarray*}
\end{enumerate}
\item \begin{enumerate}
\item Suppose $\mathrm{Ob}(I)=\{y_1,y_2\}.$ And let $$N_n(I)_{y_i} =\{ \xymatrix{(x_0\ar[r]^{f_1}&x_1\ar[r]^{f_2}&\dots\ar[r]^{f_m}&x_n)}\in N_n(I) \mid x_n=y_i\}.$$
Put $\alpha_n=\# N_n(I)_{y_1}, \beta_n=\# N_n(I)_{y_2}$. Then, we set a simultaneous recursion
\begin{eqnarray}
\alpha_n&=&a \alpha_{n-1}+c \beta_{n-1} \label{zenka1}\\
\beta_n&=&b \alpha_{n-1} +d \beta_{n-1}.\label{zenka2}
\end{eqnarray}
Let \begin{eqnarray}\alpha_{n+1}+A \beta_{n+1}&=&B(\alpha_n+A\beta_n) \label{recursion3}\end{eqnarray}
for some constant $A,B$.
The equalities $(\ref{zenka1}),(\ref{zenka2})$ imply
\begin{eqnarray*}
a\alpha_{n}+c\beta_{n}+A(b\alpha_n+d\beta_n)&=&B(\alpha_n+A \beta_n) \\
(a+Ab)\alpha_n+(c+Ad)\beta_n&=&B\alpha_n +AB \beta_n.
\end{eqnarray*}
Hence, we obtain
\begin{eqnarray*}
a+bA&=&B \\
c+dA&=&AB.
\end{eqnarray*}
So we obtain the equation $bA^2+(a-d)A-c=0$. Therefore, we have
$$A^{\pm}=\frac{(d-a)\pm\sqrt{(d-a)^2+4bc}}{2b}.$$
and $$B^{\pm}=\frac{(d+a)\pm\sqrt{(d-a)^2+4bc}}{2}.$$

For $(A^+,B^+)$, the recursion (\ref{recursion3}) is 
$$\alpha_{n+1}+A^+ \beta_{n+1}=B^+(\alpha_n+A^+\beta_n).$$ For $n=1$ we have $\alpha_1+A^+\beta_1=(a+c)+(b+d)A^+$.
Hence, we obtain 
\begin{eqnarray}
\alpha_n+A^+\beta_n&=&((a+c)+(b+d)A^+)(B^+)^{n-1}.\label{zenka3}
\end{eqnarray}
In the same way, we obtain
\begin{eqnarray}
\alpha_n+A^-\beta_n&=&((a+c)+(b+d)A^-)(B^-)^{n-1}\label{zenka4}
\end{eqnarray}
for $(A^-,B^-)$.
Hence, we have
\begin{eqnarray*}
(\ref{zenka3})-(\ref{zenka4})&=&(A^+-A^-)\beta_n \\
&=&D^+(B^+)^{n-1}-D^-(B^-)^{n-1}
\end{eqnarray*}
where $D^{\pm}=(a+c)+(b+d)A^{\pm}.$
Thus we have $$\beta_n=\frac{1}{A^+-A^-}\left( D^+(B^+)^{n-1}-D^-(B^-)^{n-1}\right).$$
Here, we note that $$A^+-A^-=\frac{\sqrt{(d-a)^2+4bc}}{b}\not =0.$$
If we are under the assumption of (\ref{2a}), $A^+ -A^- \not = 0$.
Next, we compute $\alpha_n$. By $(\ref{zenka4})\times A^+ -(\ref{zenka3})\times A^- 
$ we obtain $$\alpha_n=\frac{1}{A^+-A^-}\left( A^+ D^-(B^-)^{n-1} -A^- D^+(B^+)^{n-1} \right).$$
Thus, \begin{multline}\# N_n(I)=\alpha_n+\beta_n= \\ \frac{1}{A^+-A^-}\left( (B^+)^{n-1}D^+(1-A^-) -(B^-)^{n-1}D^-(1-A^+)\right).\end{multline}
We obtain \begin{multline*} \zeta_{I}(z)=\exp \biggl( \sum^{\infty}_{m=1}\frac{1}{m} \frac{1}{A^+-A^-} \biggl( (B^+)^{n-1}D^+(1-A^-) \\ -(B^-)^{n-1}D^-(1-A^+) \biggr) z^m \biggr) \end{multline*}
\begin{multline}\label{thm2}
=\frac{1}{(1-B^+ z)^{\frac{D^+(1-A^-)}{B^+(A^+-A^-)}}}\frac{1}{(1-B^- z)^{-\frac{D^-(1-A^+)}{B^-(A^+-A^-)}}}.
\end{multline}
Lemma \ref{DB} implies $(\ref{thm2})$ is 
\begin{multline*}
=\frac{1}{(1-B^+ z)^{\frac{(1+A^+)(1-A^-)}{(A^+-A^-)}}}\frac{1}{(1-B^- z)^{-\frac{(1-A^+)(1-A^+)}{(A^+-A^-)}}}
\end{multline*}
\begin{multline*}
=\frac{1}{(1-B^+ z)^{1+\frac{b+c}{\sqrt{(d-a)^2+4bc}}}} \frac{1}{(1-B^- z)^{1-\frac{b+c}{\sqrt{(d-a)^2+4bc}}}}.
\end{multline*}
\item The same argument above implies $$A^{\pm}=\frac{(d-a)\pm\sqrt{(d-a)^2+4bc}}{2b}$$and$$B^{\pm}=\frac{(d+a)\pm\sqrt{(d-a)^2+4bc}}{2}.$$
Since $\det A_I=0$, we have
\begin{eqnarray*}
A^{\pm}&=&\frac{(d-a)\pm\sqrt{(d-a)^2+4bc}}{2b} \\
&=&\frac{(d-a)\pm\sqrt{(d+a)^2-4\det A_I}}{2b} \\
&=&\frac{d}{b} \text{ or } -\frac{a}{b}.
\end{eqnarray*}
Since $a+A^{\pm}b=B^{\pm}$, we obtain $$(A^+,B^+)=(\frac{d}{b}, a+d), (A^-,B^-)=(-\frac{a}{b},0).$$
For $(A^-,B^-)=(-\frac{a}{b}, 0)$, we have 
\begin{eqnarray*}
\alpha_n-\frac{a}{b}\beta_n&=&0(\alpha_{n-1}-\frac{a}{b}\beta_{n-1}) \\
b\alpha_n&=&a\beta_n.
\end{eqnarray*}
Hence, we have $$\beta_n=b\alpha_{n-1}+d\beta_{n-1}=(a+d)\beta_{n-1}.$$
We obtain $\beta_n=(b+d)(a+d)^{n-1}$ since $\beta_1=(b+d)$.
For $(A^+,B^+)=(\frac{d}{b}, a+d)$, we have 
\begin{eqnarray*}
\alpha_n+\frac{d}{b}\beta_n&=&(a+d)(\alpha_{n-1}+\frac{d}{b}\beta_{n-1}) \\
\alpha_n&=&(a+d)\alpha_{n-1}-\frac{d}{b}\beta_n +(a+d)\frac{d}{b}\beta_{n-1} \\
&=&(a+d)\alpha_{n-1}-\frac{d}{b} \bigg( (b+d)(a+d)^{n-1} \\ 
&&- (b+d)(a+d)(a+d)^{n-2} \bigg) \\
\alpha_n&=&(a+d)\alpha_{n-1}.
\end{eqnarray*}
We obtain $\alpha_n=(a+c)(a+d)^{n-1}$ since $\alpha_1=(a+c)$.
Hence, we obtain
\begin{eqnarray*}
\zeta_I(z)&=&\exp\bigg( \sum_{m=1}^{\infty} \frac{1}{m}\big( (a+c)(a+d)^{n-1}+(b+d)(a+d)^{n-1}\big)z^m\bigg) \\
&=&\frac{1}{\big(1-(a+d)z \big)^{\frac{a+b+c+d}{a+d}}}.
\end{eqnarray*}
\end{enumerate}
\end{enumerate}
\end{proof}
\end{them}

\begin{lemma}\label{DB}
Let $a,b,c,d$ be non-negative integers and suppose $a,b,d \not =0$. Let 
$$A^{\pm}=\frac{(d-a)\pm\sqrt{(d-a)^2+4bc}}{2b}$$
$$B^{\pm}=\frac{(d+a)\pm\sqrt{(d-a)^2+4bc}}{2}$$
$$D^{\pm}=(a+c)+(b+d)A^{\pm}.$$
Then, we have $$\frac{D^{\pm}}{B^{\pm}}=(1+A^{\pm}).$$
\begin{proof}
We show $D^{\pm}=(1+A^{\pm})B^{\pm}.$
\begin{eqnarray*}
(1+A^{\pm})B^{\pm}&=&(1+A^{\pm})(b A^\pm +a) \\
&=&bA^{\pm}+a+b(A^{\pm})^2+aA^{\pm} \\
&=&a+b(A^{\pm})^2+(a+b)A^{\pm} \\
&=&a+(a+b)A^{\pm}\\
&&+b\frac{(d-a)^2\pm2(d-a)\sqrt{(d-a)^2+4bc}+(d-a)^2+4bc}{4b^2} \\
&=&a+(a+b)A^{\pm}+\frac{(d-a)^2+2bc\pm(d-a)\sqrt{(d-a)^2+4bc}}{2b} \\
&=&a+(a+b)A^{\pm}+c+(d-a)\frac{(d-a)\pm\sqrt{(d-a)^2+4bc}}{2b} \\
&=&a+c+(a+b)A^{\pm}+(d-a)A^{\pm}\\
&=&(a+c)+(b+d)A^{\pm} \\
&=&D^{\pm}.
\end{eqnarray*}
\end{proof}
\end{lemma}

\begin{lemma}\label{series}
Suppose $I$ is a finite category whose adjacency matrix is $A_I=\begin{pmatrix}a&b\\c&d\end{pmatrix}$. Then, there exists $\chi_{\Sigma}(I)$ if and only if $a+d-b-c=0$ or there exists $A^{-1}$.
\begin{proof}
If there exists $A^{-1}$, $\chi_{\Sigma}(I)$ exists \cite{Leib}.
We suppose there does not exist $A^{-1}$. We have
\begin{eqnarray*}
\det \left( E-\left( \begin{pmatrix}a&b\\ c&d\end{pmatrix}-E\right)t\right)&=&1+(2-(a+d))t \\
&&+(-a-d+1+(ad-bc))t^2 \\
&=&(1+t)(1+(1-(a+d))t) \\
\mathrm{sum}\bigg( (\mathrm{adj}\left( E-\left( \begin{pmatrix}a&b\\ c&d\end{pmatrix}-E\right)t\right)\bigg) &=&2+(-a-d+b+c+2)t .
\end{eqnarray*}
Hence, if $a+d-b-c\not =0$, there does not exist $\chi_{\Sigma}(I)$. And if $a+d-b-c =0$ it does. In particular, for the latter case we obtain $\chi_{\Sigma}(I)=\frac{2}{a+d}$.
\end{proof}
\end{lemma}

\begin{them}
Suppose $I$ is a finite category and the number of objects of $I$ is 2. Then, Conjecture \ref{noguchi} holds true.
\begin{proof}
We use the same enumeration of Theorem \ref{two}.
\begin{enumerate}
\item It is clear.
\item \begin{enumerate}
\item We only show (\ref{c4}) is satisfied. The left hand of (\ref{c4}) is
\begin{eqnarray*}
\frac{1}{a}\bigg(1-\frac{b}{d-a}\bigg)+\frac{1}{d}\bigg(1+\frac{b}{d-a}\bigg) &=&\frac{d^2-a^2-bd+ab}{ad(d-a)} \\
&=&\frac{(d-a)(a-b+d)}{ad(d-a)} \\
&=&\frac{a-b+d}{ad}\\
&=&\mathrm{sum}(A_I^{-1}) \\
&=&\chi_{\Sigma}(I).
\end{eqnarray*}
\item We only show (\ref{c4}) is satisfied. The left hand of (\ref{c4}) is
$$\frac{2}{a}+\frac{b}{a^2}(-1)^1=\frac{2a-b}{a^2}=\mathrm{sum}(A_I^{-1})=\chi_{\Sigma}(I).$$
\end{enumerate}
\item \begin{enumerate}
\item At first, $B^{\pm}$ are the eigen values of $A_I$. We show (\ref{c4}) is satisfied. The left hand side of (\ref{c4}) is 
\begin{eqnarray*}
&&\frac{1}{B^+}\bigg(1+\frac{b+c}{\sqrt{(d-a)^2+4bc}}\bigg)+\frac{1}{B^-}\bigg(1-\frac{b+c}{\sqrt{(d-a)^2+4bc}}\bigg)\\
&=&\frac{1}{\det A_I}\bigg(B^-+B^+ \bigg)+\frac{(b+c)}{\det A_I(\sqrt{(d-a)^2+4bc})}\bigg( B^- -B^+\bigg) \\
&=&\frac{1}{\det A_I}\bigg(B^-+B^+ \bigg)-\frac{(b+c)(\sqrt{(d-a)^2+4bc})}{\det A_I(\sqrt{(d-a)^2+4bc})} \\
&=&\frac{a+d-b-c}{\det A_I} \\
&=&\chi_{\Sigma}(I).
\end{eqnarray*}
Note that $B^+B^-=\det A_I$.
\item Suppose the series Euler characteristic of $I$ can be defined. In this case, Lemma \ref{series} implies $a+d-b-c=0$ since $\det A_I=0$. The left hand side of (\ref{c4}) is 
\begin{eqnarray*}
\frac{a+b+c+d}{(a+d)^2}&=&\frac{2(a+d)}{(a+d)^2} \\
&=&\frac{2}{a+d} \\
&=&\chi_{\Sigma}(I).
\end{eqnarray*}
\end{enumerate}
\end{enumerate}
\end{proof}
\end{them}

\subsection{The other cases}
In this subsection we prove finite categories which satisfy certain condition holds Conjecture \ref{noguchi}. We do not have any assumption about the number of objects.
\begin{lemma}\label{row}
Let $A=(a_{ij}),B=(b_ij)$ be $(n,n)$-matrices over a commutative ring $R$. Put $AB=(c_{ij})$. If $\sum_i a_{ij}=\sum_i a_{ij'}$ and $\sum_i b_{ij}=\sum_i b_{ij'}$ for any $j,j'$, then $\sum_i c_{ik}=\sum_i c_{ik'}$ for any $k,k'$. In that case, we obtain $\sum_i c_{ij}=(\sum_i a_{ij})(\sum_i b_{ij})$.
\begin{proof}
\begin{eqnarray*}
\sum_i c_{ij}&=& \sum_i\left (\sum_k a_{ik}b_{kj}\right) \\
&=&\sum_k \left(\sum_i a_{ik}b_{kj}\right) \\
&=&\sum_k \left(b_{kj} \sum_i a_{ik}\right) \\
&=&\left( \sum_i a_{ij}\right)\left( \sum_i b_{ij}\right).
\end{eqnarray*}
Note that this calculation does not depend on $j$. 
\end{proof} 
\end{lemma}

\begin{them}
Suppose $I$ is a finite category and $A_I=(a_{ij})$ satisfies the condition $\sum_i a_{ij}=\sum_i a_{ij'}$ for any $j,j'$. Then, its zeta function satisfies the conjecture, that is, 
\begin{enumerate}
\item $\displaystyle \zeta_I(z)=\frac{1}{\left( 1-\left(\sum_i a_{ij}\right)z\right)^{\# \mathrm{Ob}(I)}}$
\item $\displaystyle\frac{\# \mathrm{Ob}(I)}{\sum_i a_{ij}}=\chi_{\Sigma}(I)$
\item $\sum_i a_{ij}$ is the eigen value of $A_I$.
\end{enumerate}
\begin{proof}
Lemma \ref{row} implies $$\# N_m(I)=\mathrm{sum}(A_I^m)=\#\mathrm{Ob}(I)(\sum_i a_{ij})^m.$$
Hence, we obtain the result of $(1)$.

The same argument above implies $$\# \overline{N_m}(I)=\mathrm{sum}((A_I-E)^m)=\#\mathrm{Ob}(I)(\sum_i a_{ij}-1)^m.$$ Hence, we have 
\begin{eqnarray*}
\sum^{\infty}_{m=0} \# \overline{N_m}(I) t^m&=& \sum^{\infty}_{m=0} \mathrm{sum}((A_I-E)^m) t^m\\
&=&\sum^{\infty}_{m=0} \#\mathrm{Ob}(I)(\sum_i a_{ij}-1)^m t^m \\
&=&\frac{\# \mathrm{Ob}(I)}{1-(\sum_i a_{ij}-1)t}.
\end{eqnarray*}
So we obtain $\chi_{\Sigma}(I)=\frac{\# \mathrm{Ob}(I)}{\sum_ia_{ij}}$.
\end{proof}
\end{them}

This result is with respect to the column of $A_I$, but it is clear that this has the similar fact with respect to the row of $A_I$.

\begin{them}
Suppose $I$ is a finite category and $A_I=(a_{ij})$ satisfies the condition $\sum_j a_{ij}=\sum_j a_{i'j}$ for any $i,i'$. Then, its zeta function satisfies the conjecture, that is, 
\begin{enumerate}
\item $\displaystyle \zeta_I(z)=\frac{1}{\left( 1-\left(\sum_j a_{ij}\right)z\right)^{\# \mathrm{Ob}(I)}}$
\item $\displaystyle\frac{\# \mathrm{Ob}(I)}{\sum_j a_{ij}}=\chi_{\Sigma}(I)$
\item $\sum_j a_{ij}$ is the eigen value of $A_I$.
\end{enumerate}
\end{them}


\subsection{Examples}\label{many-exam}
In this subsection, we take four examples. All of them have three objects and are computed by using recursions.

Suppose $I$ is a finite category and $\mathrm{Ob}(I)=\{x_1,x_2,x_3\}$ and $A_I=(h_{ij})$. Put 
$$\#N_n(I)_{x_1}=a_n, \#N_n(I)_{x_2}=b_n,\#N_n(I)_{x_3}=c_n.$$Then, we have
\begin{eqnarray*}
a_n&=&h_{11}a_{n-1}+h_{21}b_{n-1}+h_{31}c_{n-1}  \\
b_n&=&h_{12}a_{n-1}+h_{22}b_{n-1}+h_{32}c_{n-1}  \\
c_n&=&h_{13}a_{n-1}+h_{23}b_{n-1}+h_{33}c_{n-1}.  
\end{eqnarray*} We set a recursion
$$a_n+yb_n+zc_n=x(a_{n-1}+yb_{n-1}+zc_{n-1})$$
for some $x,y,z$. 
Therefore, we have
$$A_I\begin{pmatrix}1\\y\\z\end{pmatrix}=x\begin{pmatrix}1\\y\\z\end{pmatrix}.$$
We can obtain the zeta functions by solving this eigen equation.

\begin{exam}
Let $I$ be a finite category whose adjacency matrix is $\begin{pmatrix}2&3&5\\2&3&5\\2&1&3\end{pmatrix}$. 
The existence of such category is assured by Lemma 4.1 \cite{Leib}. Then, $\chi_L(I)$ and $\chi_{\Sigma}(I)$ are not defined \cite{Leib}.
We have $$\det(A_I-Ex)=-x^2(x-8).$$

We have
\begin{eqnarray*}
a_n&=&2a_{n-1}+2b_{n-1}+2c_{n-1} \\
b_n&=&3a_{n-1}+3b_{n-1}+c_{n-1} \\
c_n&=&5a_{n-1}+5b_{n-1}+3c_{n-1} 
\end{eqnarray*}
We obtain two solutions $(x,y,z)=(0,1,-1), (8,1,\frac{3}{5})$. For $(x,y,z)=(0,1,-1)$, we have 
$a_n+b_n-c_n=0$. Hence, we obtain 
\begin{eqnarray}
a_n+b_n&=&c_n.\label{x1}
\end{eqnarray}
For $(x,y,z)=(8,1,\frac{3}{5})$, the equation $(\ref{x1})$ implies
\begin{eqnarray*}
a_n+b_n+\frac{3}{5}c_n &=&8(a_{n-1}+b_{n-1}+\frac{3}{5}c_{n-1}) \\
\frac{8}{5}c_n&=&8(\frac{8}{5}c_{n-1}) \\
c_n&=&8c_{n-1}.
\end{eqnarray*}
Since $c_1=13$, we obtain $c_n=13\times 8^{n-1}$.
Hence, we have $$\#N_n(I)=a_n+b_n+c_n=2c_n=\frac{13}{4}8^n.$$
we obtain $$\zeta_I(z)=\frac{1}{(1-8z)^{\frac{13}{4}}}.$$
We note that the index is not the number of objects of $I$, that is, $\frac{13}{4}\not =\frac{12}{4}=3$. Therefore, the existence of the series Euler characteristic influences the other statements of Conjecture \ref{noguchi}.
\end{exam}

\begin{exam}
Let $I$ be a finite category whose adjacency matrix is $\begin{pmatrix}2&2&2\\2&2&2\\2&8&5\end{pmatrix}$. Then, both $\chi_L(I)$ and $\chi_{\Sigma}(I)$ are defined \cite{Leib}.
We have $$\chi_L(I)=\frac{1}{2}, \chi_{\Sigma}(I)=\frac{1}{3}$$ $$\det(A_I-Ex)=-x^2(x-9)$$
and
\begin{eqnarray*}
a_n&=&2a_{n-1}+2b_{n-1}+2c_{n-1} \\
b_n&=&2a_{n-1}+2b_{n-1}+2c_{n-1} \\
c_n&=&2a_{n-1}+8b_{n-1}+5c_{n-1} 
\end{eqnarray*}
We obtain two solutions $(x,y,z)=(0,1,-2), (9,1,\frac{5}{2})$. For $(x,y,z)=(0,1,-2)$, we have 
$a_n+b_n-2c_n=0$. Hence, we obtain 
\begin{eqnarray}
a_n+b_n&=&2c_n.\label{y1}
\end{eqnarray}
For $(x,y,z)=(9,1,\frac{5}{2})$, the equation $(\ref{y1})$ implies
\begin{eqnarray*}
a_n+b_n+\frac{5}{2}c_n &=&9(a_{n-1}+b_{n-1}+\frac{5}{2}c_{n-1}) \\
\frac{9}{2}c_n&=&9(\frac{9}{2}c_{n-1}) \\
c_n&=&9c_{n-1}.
\end{eqnarray*}
Since $c_1=9$, we obtain $c_n=9^n$.
Hence, we have $$\#N_n(I)=a_n+b_n+c_n=3c_n=3\times 9^n.$$
we obtain $$\zeta_I(z)=\frac{1}{(1-9z)^3}.$$
For $(\ref{c4})$, we obtain $\frac{3}{9}=\chi_{\Sigma}(I)$.
\end{exam}

The zeta functions of the following two examples use complex numbers.

\begin{exam}\label{complex1}
Let $I$ be a finite category whose adjacency matrix is $\begin{pmatrix}2&3&2 \\1&2&6 \\1&1&2\end{pmatrix}$.  We have 
$$\det A_I=6, A_I^{-1}=\frac{1}{6}\begin{pmatrix}-2&-4&14 \\4&2&-10 \\-1&1&1\end{pmatrix}$$
$$\det (A_I -Ex)=-(x+i)(x-i)(x-6).$$
We have three solutions $$(x,y,z)=(6,1,\frac{1}{2}),(i,-\frac{16}{25}+\frac{13}{25}i,-\frac{1}{25}-\frac{7}{25}i),(-i,-\frac{16}{25}-\frac{13}{25}i,-\frac{1}{25}+\frac{7}{25}i).$$
For $(x,y,z)=(6,1,\frac{1}{2})$, we have
$$a_n+b_n+\frac{1}{2}c_n=6(a_{n-1}+b_{n-1}+\frac{1}{2}c_{n-1}).$$
Since we have $a_1+b_1+\frac{1}{2}c_1=4+6+\frac{1}{2}\times 10=15$, we obtain 
\begin{eqnarray}a_n+b_n+\frac{1}{2}c_n=15\times 6^{n-1}.\label{i1}\end{eqnarray}
In the same way, for $(x,y,z)=(i,-\frac{16}{25}+\frac{13}{25}i,-\frac{1}{25}-\frac{7}{25}i)$, we obtain
\begin{eqnarray}
a_n+(-\frac{16}{25}+\frac{13}{25}i)b_n+(-\frac{1}{25}-\frac{7}{25}i)c_n&=&\frac{-6+8i}{25}i^{n-1}.\label{i2}
\end{eqnarray}
For $(x,y,z)=(-i,-\frac{16}{25}-\frac{13}{25}i,-\frac{1}{25}+\frac{7}{25}i)$, we obtain
\begin{eqnarray}
a_n+(-\frac{16}{25}-\frac{13}{25}i)b_n+(-\frac{1}{25}+\frac{7}{25}i)c_n&=&\frac{-6-8i}{25}(-i)^{n-1}.\label{i3}
\end{eqnarray}
By $(\ref{i2})-(\ref{i3})$ we have
\begin{eqnarray*}
-13b_n+7c_n&=&i^n\{ (-3+4i)-(-1)^n(3+4i)\}.
\end{eqnarray*}
By $(\ref{i3})-(\ref{i1})$ we have
\begin{eqnarray*}
82b_n+27c_n&=&-2(6+8i)(-i)^{n+1}+ \\
&&125\times 6^n+2i^{n+1}\{ (-3+4i)+(-1)^{n+1}(3+4i) \}.
\end{eqnarray*}
Hence, we have 
\begin{multline}
c_n=\frac{1}{925}\bigg( -26(6+8i)(-i)^{n+1}+5^313 \times 6^n \\ +2\times 13 i^{n+1}\big((-3+4i)+(-1)^{n+1}(3+4i) \big)   \\ +82i^n\big( (-3+4i)+(-1)^{n+1}(3+4i)\big)\bigg).
\end{multline}
Hence, we obtain
\begin{eqnarray*}
\# N_n(I)&=&a_n+b_n+c_n \\
&=&a_n+b_n+\frac{1}{2}c_n+\frac{1}{2}c_n \\
&=&15\times 6^{n-1}+\frac{1}{925}\bigg( -26(3+4i)(-i)^{n+1}+\frac{1}{2}5^313\times 6^n \\
&&+13i^{n+1}  \big((-3+4i) +(-1)^{n+1} (3+4i)\big) \\ 
&&+41i^n\big( (-3+4i)+(-1)^{n+1}(3+4i)\big)\bigg) \\
&=&\frac{125}{37}6^n+\frac{1}{37}\bigg( i^n(-7+5i)-(-i)^n(7+5i)\bigg).
\end{eqnarray*}
Finally, we obtain the zeta function of $I$
$$\zeta_I(z)=\frac{(1+iz)^{\frac{7+5i}{27}}}{(1-6z)^{\frac{125}{37}} (1-iz)^{\frac{-7+5i}{37}}}.$$
Conjecture \ref{noguchi} holds true for this zeta function.
\begin{enumerate}
\item The sum of the indices of the zeta function is the number of objects of $I$
$$\frac{125}{37}-\frac{7+5i}{37}+\frac{-7+5i}{37}=3.$$
\item For (\ref{c4}) we obtain
$$\frac{125}{37\times 6}-\frac{7+5i}{37}\frac{1}{-i}+\frac{-7+5i}{37}\frac{1}{i}=\frac{5}{6}=\mathrm{sum}(A_I^{-1})=\chi_{\Sigma}(I).$$
\item $6,i,-i$ are the eigen value of $A_I$.
\end{enumerate}
\end{exam}

\begin{exam}\label{complex2}
Let $I$ be a finite category whose adjacency matrix is $\begin{pmatrix}4&7&8 \\1&4&5 \\1&1&3\end{pmatrix}$.  We have 
$$\det A_I=18, A_I^{-1}=\frac{1}{18}\begin{pmatrix}7&-13&3 \\ 2&4&-12 \\-3&3&9\end{pmatrix}$$
and $$\det(A_I-Ex)=-(x-9)(x-\lambda)(x-\bar{\lambda})$$
where $\lambda=(1+i)$ and $\bar{\lambda}=(1-i)$.
We have three solutions $$(x,y,z)=(9,\frac{11}{25},\frac{6}{25}),(\lambda,\frac{-1+3i}{5},\frac{-1-2i}{5}),(\bar{\lambda},\frac{-1-3i}{5},\frac{-1+2i}{5}).$$
For $(x,y,z)=(9,\frac{11}{25},\frac{6}{25})$, we have
$$a_n+\frac{11}{25}b_n+\frac{6}{25}c_n=9(a_{n-1}+\frac{11}{25}b_{n-1}+\frac{6}{25}c_{n-1}).$$
Since we have $a_1+\frac{11}{25}b_1+\frac{6}{25}c_1=6+\frac{11}{25}\times 12 +\frac{6}{25}\times 16=\frac{378}{25}$, we obtain 
\begin{eqnarray}a_n+\frac{11}{25}b_n+\frac{6}{25}c_n=\frac{42}{25}\times 9^{n}.\label{j1}\end{eqnarray}
In the same way, for $(x,y,z)=(\lambda,\frac{-1+3i}{5},\frac{-1-2i}{5})$, we obtain
\begin{eqnarray}
a_n+\left(\frac{-1+3i}{5}\right)b_n+\left(\frac{-1-2i}{5}\right)c_n&=&\frac{3+i}{5}\lambda^n.\label{j2}
\end{eqnarray}
For $(x,y,z)=(\bar{\lambda},\frac{-1-3i}{5},\frac{-1+2i}{5})$, we obtain
\begin{eqnarray}
a_n+\left(\frac{-1-3i}{5}\right)b_n+\left(\frac{-1+2i}{5}\right)c_n&=&\frac{3-i}{5}{\bar{\lambda}}^n.\label{j3}
\end{eqnarray}

By the simultaneous equations $(\ref{j1}),(\ref{j2}),(\ref{j3})$, we obtain
\begin{eqnarray*}
a_n&=&\left( \frac{23+11i}{130}\right)\lambda^n+\left( \frac{23-11i}{130}\right){\bar{\lambda}}^n+\frac{42}{65}9^n \\
b_n&=&\left( \frac{-19-43i}{130}\right)\lambda^n+\left( \frac{-19+43i}{130}\right){\bar{\lambda}}^n+\frac{84}{65}9^n \\
c_n&=&\left( \frac{-61+33i}{130}\right)\lambda^n+\left( \frac{-61-33i}{130}\right){\bar{\lambda}}^n+\frac{126}{65}9^n 
\end{eqnarray*}
Hence, we obtain $$\#N_n(I)=a_n+b_n+c_n=\left( \frac{-57+i}{130}\right)\lambda^n+\left( \frac{-57-i}{130}\right){\bar{\lambda}}^n+\frac{252}{65}9^n 
$$
Finally, we obtain the zeta function of $I$
$$\zeta_I(z)=\frac{1}{(1-9z)^{\frac{252}{65}} (1-\lambda z)^{\frac{-57+i}{130}} (1-\bar{\lambda} z)^{\frac{-57-i}{130}}}.$$
Conjecture \ref{noguchi} holds true for this zeta function.
\begin{enumerate}
\item The sum of the indices of the zeta function is the number of objects of $I$
$$\frac{252}{65}+\frac{-57+i}{130}+\frac{-57-i}{130}=3.$$
\item For (\ref{c4}) we obtain
$$\frac{252}{65\times 9}+\frac{-57-i}{130}\frac{1}{\bar{\lambda}}+\frac{-57+i}{130}\frac{1}{\lambda}=0=\mathrm{sum}(A_I^{-1})=\chi_{\Sigma}(I).$$
\item $9,\lambda,\bar{\lambda}$ are the eigen value of $A_I$.
\end{enumerate}
\end{exam}

\section{The zeta function of directed graphs}\label{graph}

In this section, we consider the relation between the zeta function of directed graphs and our zeta function. The two notions, directed graphs and small categories, are very similar. They are consisted by the set of vertices (objects) and the set of arrows (morphisms). The most difference point is that a composition of arrows is defined or not.
Here, we have two functors between the category of directed graphs $\mathbf{DG}$ and the category of small categories $\mathbf{Cat}$
$$\xymatrix{\mathbf{DG}\ar@<1ex>[r]^F&\mathbf{Cat}\ar@<1ex>[l]^Q}$$
where $F$ is the functor of free categories and $Q$ is the forgetful functor \cite{ML98}.

For a directed graph $D$, the zeta function $Z_D(u)$ of $D$ is defined by the formal product of certain equivalence class of paths, see \cite{MS01} for more details. It has the determinant expression of the following form
$$Z_D(u)=\frac{1}{\det (E-Au)}$$
where $A$ is the adjacency matrix of $D$.

We apply the notion "acyclic" to directed graphs. A directed graph is \textit{acyclic} if there is no oriented arrows and there is no arrows from a vertex $y$ to a vertex $x$ if there exists an arrow from $x$ to $y$. A directed graph becomes a finite category if and only if $D$ is finite acyclic.

\begin{prop}
Let $D$ be a finite acyclic directed graph. Then, we obtain
$$\zeta_{F(D)}(z)=Z_D(z)\exp\left( \sum_{k=1}^{\infty} \frac{\# \overline{N_k}(F(D)) }{k}(z^{-1}-1)^{-k}  \right).$$
\begin{proof}
The zeta function of $D$ is $$\frac{1}{(1-z)^{\# V(D)}}$$ since $A_D$ is an upper triangular matrix whose diagonal entries are all 1. Hence, Proposition \ref{non-degenerate} implies this result.
\end{proof}
\end{prop}

The zeta function of a finite category has the integral expression and the zeta function of directed graph appears in the expression, that is,

\begin{prop}
Let $I$ be a finite category. Then, we have
\begin{eqnarray*}
\zeta_I(z)&=&z^{-\#N_0(I)} \exp\left(   \int\frac{1}{z}\frac{\mathrm{sum}(\mathrm{adj}(E-(A_I-E)\frac{1}{z^{-1}-1}))}{\det(E-(A_I-E)\frac{1}{z^{-1}-1})} dz \right) \\
&=&z^{-\#N_0(I)} \exp\left(   \int\frac{1}{z}\frac{\mathrm{sum}(\mathrm{adj}(E-(A_I-E)\frac{1}{z^{-1}-1}))}{Z_{Q(I)}(\frac{1}{z^{-1}-1})} dz \right)
\end{eqnarray*}
\begin{proof}
We start from the expression of Proposition \ref{non-degenerate}. We have
\begin{eqnarray*}
\frac{d}{dz} \log \zeta_I(z)&=&\frac{d}{dz}\left( -\# N_0(I)\log(1-z) +\left( \sum_{k=1}^{\infty} \frac{\# \overline{N_k}(I) }{k}(z^{-1}-1)^{-k}  \right) \right) \\
&=&\frac{N_0(I)}{(1-z)} +\bigg(\frac{1}{z(1-z)}\bigg)\sum_{k=1}^{\infty} \overline{N_k}(I)\bigg(\frac{z}{1-z}\bigg)^{k} \\
&=&\frac{1}{1-z}\left( \#N_0(I) +\frac{1}{z}\bigg( \sum_{k=0}^{\infty} \overline{N_k}(I)\bigg(\frac{z}{1-z}\bigg)^{k} -\#N_0(I) \bigg)\right) \\
&=&-\frac{1}{z}\#N_0 + \frac{1}{z}\frac{\mathrm{sum}(\mathrm{adj}(E-(A_I-E)(\frac{1}{z^{-1}-1})))}{\det(E-(A_I-E)(\frac{1}{z^{-1}-1}))} .
\end{eqnarray*}
Hence, we obtain the results since the adjacency matrix of $Q(I)$ is $(A_I-E)$.
\end{proof}
\end{prop}


\begin{thebibliography}{AAA99}

\bibitem [BL08]{Leib}  C. Berger and T. Leinster. The Euler characteristic of a category as the sum of a divergent series, \textit{Homology, Homotopy Appl., }10(1):41-51, 2008.

\bibitem[BS98]{BS98}   A. Bj\"orner and K. S. Sarkaria. The zeta function of a simplicial complex. Israel J. Math. 103 (1998), 29-40.


\bibitem[FLS11]{FLS} T. M. Fiore, W. L\"{u}ck and R. Sauer. Finiteness obstructions and Euler characteristics of categories, \textit{Adv. Math}, Vol. 226, Number 3, (2011), 2371--2469.

\bibitem[Har77]{Har77}  R. Hartshorne. \textit{Algebraic geometry}. Graduate Texts in Mathematics, No. 52. Springer-Verlag, New York-Heidelberg, 1977

\bibitem[Kur96]{Kur96} N. Kurokawa. Zeta functions of categories. \textit{Proc. Japan Acad. Ser.}10: 221-222 vol. 72,1996

\bibitem [Lei08]{Leia}    T. Leinster. The Euler characteristic of a category,\textit{ Doc. Math.}, 13:21-49, 2008, arXiv:math.CT/0610260

\bibitem [May99]{May99}     J. P. May. \textit{A concise course in algebraic topology} . Chicago Lectures in Mathematics. University of Chicago Press, Chicago, IL, 1999. 

\bibitem [ML98]{ML98}  Saunders Mac Lane. \textit{Categories for the working mathematician}, volume 5 of Graduate Texts in Mathematics. Springer-Verlag, New York, 1998. 

\bibitem[MS01]{MS01} H. Mizuno, I. Sato. Zeta functions of digraphs, Linear Algebra Appl. 336 (2001) p181-190

\bibitem [Nog11]{Nog11} K. Noguchi. The Euler characteristic of acyclic categories. Kyushu Journal of Mathematics, vol. 65 No.1 (2011), 85-99.

\bibitem [Nog]{Nog} K. Noguchi. The Euler characteristics of categories and the barycentric subdivision. arXiv:1104.3630

\bibitem[ST96]{ST96} H.M. Stark, A.A. Terras, Zeta functions of finite graphs and coverings,\textit{Adv. in Math.} 121 (1996), 124-165

\end{thebibliography}
\end{document}